\documentclass{aptpub}
\usepackage{tikz}

\newcommand{\pr}[1]{\mathbb{P}\!\left(#1\right)}

\authornames{Evita Nestoridi} 
\shorttitle{A non-local random walk on the hypercube} 

\begin{document}
\title{A non-local random walk on the hypercube} 

\authorone[Princeton University]{Evita Nestoridi} 

\addressone{Department of Mathematics\\
Fine Hall, Washington Road,\\
Princeton, NJ 08544-1000, USA\\
\email{exn@princeton.edu}} 

\begin{abstract}
This paper studies the random walk on the hypercube $(\mathbb{Z}/2\mathbb{Z})^n$ which at each step flips $k$ randomly chosen coordinates. We prove that the mixing time for this walk is of order $\frac{n}{k} \log n$. We also prove that if $k=o(n)$, then the walk exhibits cutoff at $\frac{n}{2k} \log n$ with window $\frac{n}{2k} $.
\end{abstract}



\section{Introduction}
Consider two urns, one containing zero balls and the second containing $n$ balls. At each step,  pick $k$ total balls at random and move  each of them to the opposite urn. This is a generalization of the Ehrenfest's urn model, which works for $k=1$.

The above Markov chain can be also viewed as a random walk on $(\mathbb Z/2\mathbb Z)^n$ where at each step we flip $k$ random coordinates (for some fixed $k$).  For the walk to be transitive, $k$ needs to be  odd, and to avoid parity problems, it is simplest to consider the lazy version of this walk. In other words, at each step, do nothing with probability $1/2$ and with probability $1/2$  choose a random set of $k$ coordinates and flip them. The main question is to find the mixing time of this walk for the total variation distance. 

 This non-local walk  implies a big change at each step.
Almost all of the tools and examples developed over the past years give careful rates for local Markov chains, such as nearest neighbor random walks on graphs and the usual implementation of the Gibbs sampler and Metropolis algorithm. On the other hand, there are a host of much more global algorithms such as Swedsen-Wang, Wolff, hit and run and, in card shuffling, the riffle shuffles and hyperplane arrangement walks, where the chain moves quite far in one step. See Andersen and Diaconis \cite{AndersenDiaconis} for a survey of many such algorithms. Cheeger and path arguments are not set up to deal with these non-local Markov chains. The present paper gives a careful study of this kind of long range walk via techniques capable of giving sharp results.

There is a second reason why this particular random walk is interesting. Fix $x \in (\mathbb Z/2\mathbb Z)^n$. Let $P$ denote the transition matrix of the random walk and let $P^{*\ell}$ denote the $\ell$th power of $P$. There are two different approaches to finding the mixing time of this walk. The first approach is developed in Section \ref{eig}. It involves finding the eigenvalues of the walk using representation theory and using the Fourier transform to give bounds on the $\ell^2$ norm of the difference $P^{*\ell}_{x}-U$, where $U$ is the uniform measure and $P^{*\ell}_{x}$ is the row of $P^{*\ell}$ corresponding to $x$. For the case of $k=1$, this technique works nicely and gives a sharp upper bound on the mixing time. However, for $k=\frac{n}{2}$, it turns out that the bound obtained via the $L^2$ norm does not give a sharp upper bound on the mixing time, which is defined in terms of the total variation distance ($L^1$ norm).

A second argument via coupling is introduced in Section \ref{coupling}. It provides a solution to the general case and makes the difference between the $L^2$ norm and total variation distance clear. This coupling argument is a generalization of one used by D.\ Aldous \cite{Aldous_correct} for the case $k=1$. See \cite{Aldous1} for more results of Aldous on the hypercube. The lower bound uses the first two eigenvectors and eigenvalues of the random walk and the second moment technique. This method was firstly introduced by P.\ Diaconis and M.\ Shahshahani in \cite{PDMS}. In their paper, they managed to prove a lower bound for the case $k=1$ that matched the Aldous' upper bound, proving in this way the existence of a cutoff at $\frac{1}{4}(n+1) \log n $. Another way to find a lower bound was proved by L.\ Saloff-Coste in \cite{LSC} using Wilson's lemma, which is Lemma $2$ of \cite{Wilson}. In \cite{PDRGM}, Diaconis, Graham and Morisson use Fourier Analysis directly to derive the exact behavior of the error for the nearest neighbor random walk.

It is easy to see that the mixing time for the $k$ model and the $n-k$ model will be the same, therefore we will focus on the case $k \leq n/2$. 
The results of this paper are the following:

\begin{theorem}\label{coupling_upper_bound}
For the lazy walk changing $k\leq n /2$ coordinates on the hypercube, the following hold for every $x\in (\mathbb Z/2\mathbb Z)^n$.
\begin{enumerate}
\item[(a)] For $\ell=  \frac{n^2}{2k(n-k) } \log n + c \frac{n^2}{k(n-k)}  $, we have that
$$\lVert P^{*\ell}_{x}-U \rVert_{T.V.} \leq e^{-c}  + 2^{-c},$$
where $c>0$.
\item[(b)] For $\ell= \frac{n}{2k}\log n -c\frac{n}{k}$, where $0<c \leq \frac{1}{4} \log n$ and for $x$ being the identity element, we have that
$$\lVert P^{*\ell}_{id} - U\rVert _{T.V.} \geq 1- \frac{B}{e^{4c}},$$
for a uniformly bounded constant $B>0$.
\end{enumerate}
\end{theorem}
The following corollary discusses cutoff.
\begin{corollary}
If $k=o(n)$, then walk exhibits cutoff at $\frac{n}{2k} \log n$ with window $\frac{n}{2k}$. 
\end{corollary}
%

 Section \ref{more} contains the analysis for  $L^2$-mixing time of the random walk on $(\mathbb{Z}/ m \mathbb{Z})^n$ generated by the measure 
\begin{equation}\label{q}
Q(a_{i_1} e_{i_1} +a_{i_2} e_{i_2}+ \dots +a_{i_k} e_{i_k}  )= \frac{1}{{n \choose k}m^k},
\end{equation}
where $\{e_j\}_{j=1}^n$ is the standard base, $a_{i_j} \in \mathbb{Z}/ m \mathbb{Z}$ and $\{ i_1, i_2, \ldots i_k\} \subset \{1,2,\dots, n\}$. Let $Q(x, xg)= Q(g)$ for every $x,g \in (\mathbb{Z}/ m \mathbb{Z})^n$.
 The main result of Section \ref{more} is:
\begin{theorem}\label{gen}
For the walk generated by $Q$, if $\ell=\frac{n+1}{2k} \log (mn) + \frac{c(n+1)}{2k}$, then for every $x \in (\mathbb{Z}/ m \mathbb{Z})^n$, we have that
$$4\lVert Q_x^{*\ell}-U\rVert ^2_{T.V.} \leq e^{-c} .$$
\end{theorem}

It is known that the total variation mixing time is faster than that (using the fact that  the first time that we have touched all coordinates  is a strong stationary time) but the above result holds for the $L^2$ norm, which allows us to use comparison theory to provide bounds for the $L^2$-mixing times of the walk on $(\mathbb{Z}/ m \mathbb{Z})^n$ generated by
\begin{equation*}
\text{\~{P}}(\pm e_1)= \frac{1}{4n}, \text{\~{P}}(id)= \frac{1}{2}.
\end{equation*}
More precisely, in Section  \ref{comp} we prove the bound 
$
m^2(\frac{n+1}{2} \log (mn) + \frac{c(n+1)}{2})$ for the mixing time of the last random walk. The analysis of the $L^2$ norm of the last walk has already been done by Diaconis and Saloff-Coste \cite{Compare}, where they proved an upper bound of order $m^2n \log n$ and, then, Saloff-Coste proved the cutoff \cite{Lbook}.

\section{The history of the Ehrenfest's urn model}
The Ehrenfest's urn model was introduced by Tatjana and Paul Ehrenfest \cite{Ehr} to study the second law of thermodynamics. This is a model for $n$ particles distributed in two containers and each particle changes container independently from the others (see Figure \ref{fi}). This process is repeated several times and the question is to find the limiting distribution of the process. M.\ Kac \cite{Kac} approached this problem by finding the eigenvectors and eigenvalues of the transition matrix. He also proved that if the initial system state is not at equilibrium then the entropy is increasing.

\begin{figure}[ht!]
\centering
\includegraphics[scale=0.4]{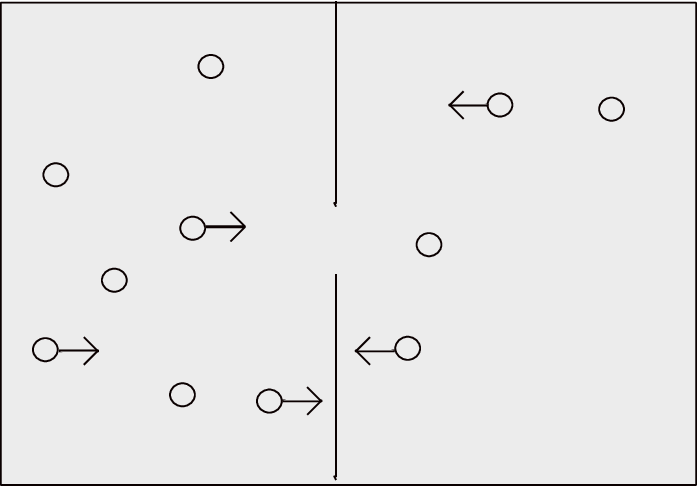}
\caption{Eleven particles in two containers, five of each changing containers}
\label{fi}
\end{figure}

This problem can also be viewed as a random walk on $(\mathbb{Z}/2\mathbb{Z})^n$ where the number of ones in the binary vector represent the number of particles in the right hand container. Flipping one (or $k$) coordinates of the binary vector corresponds to moving one (or $k$) particles to the other container. But now the Markov chain problem can be studied through a random walk on an abelian group, where representation theory is quite simple to use. As Persi Diaconis writes in Chapter $3$ of his book \cite{Pbook}, Kac  posed the question: When can a Markov chain be lifted to a random walk on a group?

\section{Coupling Argument}\label{coupling}

Consider the following measure on $(\mathbb{Z}/2\mathbb{Z})^n$ 

$$P(g)= \begin{cases} \frac{1}{2}, & \mbox{if } g=\mathrm{id} \\ \frac{1}{2 {n \choose k}}, & \mbox{if } g \in(\mathbb{Z}/2\mathbb{Z})^n \mbox{ has } k \mbox{ ones and } n-k \mbox{ zeros} \end{cases}$$
and notice that $P(x, xg)= P(g)$ for every $x,g \in (\mathbb{Z}/2\mathbb{Z})^n$.

Here is the coupling argument which will provide an upper bound for the mixing time for $k \leq \frac{n}{2}$:
Start with two different copies of the Markov chain. At time $t$ denote the state of each as $X^t_1$ and $X^t_2$.
$X_1$ will start at a deterministic $x \in (\mathbb{Z}/2\mathbb{Z})^n$ and $X_2$ will start at a random configuration.
At time $t$, let 
\begin{equation}\label{y}
y(t)= \lVert X^t_1- X^t_2\rVert _1=\sum_i|X^t_{1}(i)- X^t_{2}(i)|,
\end{equation} 
where $X^t(i)$ denotes the $i$th coordinate of the corresponding vertor.
Then consider the following cases:

\begin{enumerate}
\item If $y(t)$ is odd then take one independent step on each chain according to the probability measure $P$.
\item If $y(t)$ is even then  with probability $\frac{1}{2}$ stay fixed in both chains. With probability $\frac{1}{2 {n \choose k}}$ choose $k$ coordinates and denote them by $i_1<\ldots < i_k$. Flip $X^t_1(i_1),\ldots, X^t_1(i_k)$. To determine the move on $X^t_2$, we need the following definition.
\begin{definition}\label{a}
Denote by $a(t)$ the number of the mismatching coordinates among the $k$ ones selected.
\end{definition}
 If $a(t)> \frac{y(t)}{2}$ then flip $X^t_2(i_1),\ldots, X^t_2(i_k)$. If $a(t) \leq  \frac{y(t)}{2}$, for $i_1$ find the first $i_1' \notin \lbrace i_1, \ldots, i_k \rbrace$ such that $X^t_2(i_1')\neq X^t_1(i_1') $. Flip $X^t_2(i_1')$. Then, find the second $i_2'\notin \lbrace i_1, \ldots, i_k \rbrace$ (cyclically) such that $X^t_2(i_2')\neq X^t_1(i_2') $, flip $X^t_2(i_1')$, etc. This way we determine which coordinates of $X^t_2$ will be flipped.

\begin{figure}[h!]
\centering
\includegraphics[scale=0.33]{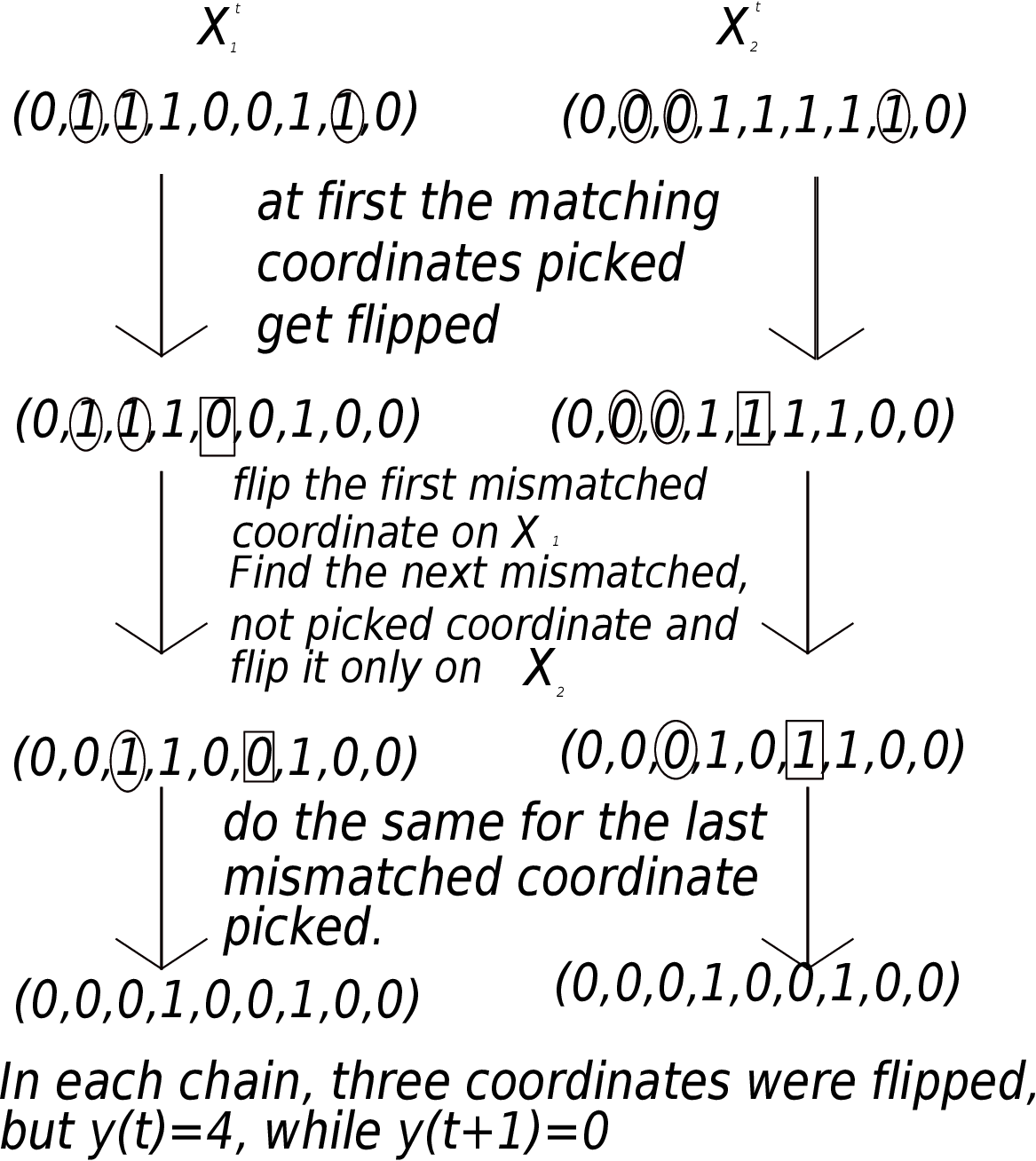}
\caption{This picture gives an example of how the coupling works}
\end{figure}

\end{enumerate}
\begin{definition}
If $x\in (\mathbb{Z}/ 2 \mathbb{Z})^n$ is the starting configuration of $X_1$, then let $T_x$ denote the first time the two chains match.
\end{definition}
The following lemma, which can be found as Lemma $5$ of Chapter $4$ of \cite{Pbook}, says how  the above coupling can be used to get an upper bound for the total variation distance.

\begin{lemma}\label{theory}
For every starting configuration $x \in (\mathbb{Z}/ 2 \mathbb{Z})^n$, we have that
$$\lVert P_x^{*\ell}-U\rVert _{T.V.} \leq P(T_x >\ell).$$
\end{lemma}

\begin{remark}
The proof of Theorem \ref{coupling_upper_bound} does not depend on the starting configuration, since we will be bounding the quantity $\max_x P(T_x >\ell)$.
\end{remark}

\section{Proof of Theorem \ref{coupling_upper_bound}}
\subsection{Upper Bound}
\begin{proof}[Proof of \ref{coupling_upper_bound}(a).]
At first, in case that the starting number of the mismatched coordinates of the chains is odd, wait until the coupling suggests staying fixed at one of them and take a step on the other to turn the difference even. Call $t_0$ the time the above happens. Then $t_0$ follows  a geometric distribution with probability of success $\frac{1}{2}$. Therefore,
$$\pr{t_0 >c \frac{n^2}{2k(n-k)} }  \leq \frac{1}{2^{c \frac{n^2}{2k(n-k)} } }\leq 2^{-c}.$$

The coupling ensures that the distance remains even. To analyze the coupling, we imitate path coupling techniques as they were introduced by Bubley and Dyer \cite{pathcoupl}. As explained in Section $2$ of \cite{pathcoupl}, we need to consider a new graph structure on $(\mathbb{Z}/2 \mathbb{Z})^n$ and examine how the distance behaves after one step if $X_1^{\ell}$ and $X_2^{\ell}$ are in adjacent positions.

Let $p(x,y)= \frac{\lVert  x-y \rVert}{2}$ and $x$ be adjacent to $y$ if and only if $p(x,y)=1$.
If $X_1^{\ell-1}$ and $X_2^{\ell-1}$ differ only at two coordinates, that is $p( X_1^{\ell}, X_2^{\ell})=1$, then  the probability of having that $X_1^{\ell}=X_2^{\ell}$ is $\frac{2k(n-k)}{n(n-1)}$, which gives that
\begin{equation*}
\mathbb{E}\{ \lVert  X_1^{\ell}- X_2^{\ell} \rVert \vert X_1^{\ell-1},X_2^{\ell-1} \} \leq \left( 1- \frac{2k(n-k)}{n(n-1)}\right) \lVert  X_1^{\ell-1}-X_2^{\ell-1} \rVert.
\end{equation*}
Path coupling says that this extends to
\begin{equation*}
\mathbb{E}\{ \lVert  X_1^{\ell}-X_2^{\ell} \rVert \vert X_1^{t_0},X_2^{t_0} \} \leq n \left( 1- \frac{2k(n-k)}{n(n-1)}\right)^{\ell-t_0}
\end{equation*}
where $p( X_1^{\ell}, X_2^{\ell})$ is not necessarily $1$, as explained thoroughly in Theorem $6.2$ of \cite{Sinclair}.
     
If $\ell=\frac{n^2}{2k(n-k) } \log n + c \frac{n^2}{k(n-k)} $, Markov's inequality gives that
\begin{align*}
&\pr{T>\ell} = \pr{T>\ell \vert t_0 \leq c \frac{n^2}{2k(n-k)} } \pr{t_0 \leq c \frac{n^2}{2k(n-k)} } \\
&+\pr{T>\ell \vert t_0 >c \frac{n^2}{2k(n-k)} } \pr{t_0> c \frac{n^2}{2k(n-k)} } \\
&\leq\pr{ \lVert  X_1^{\ell}-X_2^{\ell} \rVert \geq 1 \vert t_0 \leq c \frac{n^2}{2k(n-k)} } + \pr{t_0>c \frac{n^2}{2k(n-k)} } \\
& \leq \mathbb{E}\{ \lVert  X_1^{\ell}-X_2^{\ell} \rVert \vert X_1^{0},X_2^{0}, t_0 \leq c \frac{n^2}{2k(n-k)} \}+2^{-c}  
\end{align*}
\begin{align*}
 \leq n \left( 1- \frac{2k(n-k)}{n(n-1)}\right)^{\ell-c \frac{n^2}{2k(n-k)} } +2^{-c}\leq e^{-c} + 2^{-c}.
\end{align*}
\end{proof}
\subsection{Lower Bound}
\label{sec:lower}
The lower bound will be proved using the eigenvectors and eigenvalues for this Markov chain. In this section, the  random walk will begin at the identity. Theorem $6$ of \cite{Pbook} (page $49$) says that the eigenvalues are the Krawtchuck polynomials and the eigenvectors are the normalized Krawtchuck polynomials. To see this, notice that the irreducible representations of $(\mathbb{Z}/2\mathbb{Z})^n$ are indexed by vectors $\textbf{a} \in (\mathbb{Z}/2\mathbb{Z})^n$ so that
$$\rho_{\textbf{a}}(\textbf{v})= (-1)^{\textbf{a} \cdot \textbf{v}},$$
where $\textbf{v} \in (\mathbb{Z}/2\mathbb{Z})^n$. 
Therefore, the Fourier transform of $P$ at $\rho_{\textbf{a}}$ is 
$$\widehat{P}(\rho_{\textbf{a}}):= \sum_{\textbf{v}} \rho_{\textbf{a}}(\textbf{v}) P(\textbf{v})= \frac{1}{2} +\frac{1}{2} \sum^k_{b=0}(-1)^b \frac{{j \choose b} {n-j \choose k-b}}{{n \choose k}},$$
where $j$ denotes the number of coordinates of $\textbf{a}$ that are equal to one. According to Theorem $6$ of \cite{Pbook}, the eigenvalues of the transition matrix are exactly the $\widehat{P}(\rho_{\textbf{a}})$, $  \textbf{a} \in (\mathbb{Z}/2\mathbb{Z})^n$. The corresponding (non-normalized) eigenfunction is $f_{\textbf{a}}(\textbf{x})= (-1)^{\textbf{x} \cdot \textbf{a}}$. Notice that all $\textbf{a} \in (\mathbb{Z}/2 \mathbb{Z})^n $ that have the same number of zeros give 
the same eigenvalue. Thus, if $|\textbf{x}| $ denotes the 
number of ones of $\textbf{x}$, the $j^{th}$ Krawtchuck 
polynomials $$f_j(\textbf{x})=\sum^{|\textbf{x}|}_{b=0}
(-1)^b \frac{ {|\textbf{x}| \choose b} {n-|\textbf{x}| 
\choose j-b} }{{n \choose j}}$$
 are eigenfunctions and their normalized form will be  used to compute the lower bound for the mixing time.
\begin{proof}[Proof of \ref{coupling_upper_bound}(b)]
Remember that the definition of the total variation distance is
$$\lVert P-Q\rVert_{T.V.}  = \sup_{A} |P(A)-Q(A)|.$$

A specific set $A$ will provide a lower bound.  To find this lower bound, consider the normalized Krawtchuck polynomial of degree one $f(x)= \sqrt{n} (1-\frac{2x}{n})$ and the non-normalized Krawtchuck polynomial of degree two $f_2(x)= 1 -\frac{4x}{n-1} + \frac{4x^2}{n^2-n}$.
Then, consider $A_{\alpha} = \{ x: |f(x)| \leq \alpha \}.$ A specific choice for $\alpha$
will guarantee the correct lower bound. 

The orthogonality relations that the normalized Krawtchuck polynomials satisfy give that
if $Z$ is a point chosen uniformly in $X=\{0,1,2,...,n\}$ then
$$\mathbb{E}\{ f(Z) \}=0 \mbox{ and } \operatorname{Var} \{  f(Z) \}=1.$$
If $Z_{\ell}= \lVert X_{\ell} \rVert$, then we have that 
$$ \mathbb{E} \{ f(Z_{\ell})\} = \sqrt{n} \left( \frac{1}{2}+ \frac{1}{2}\left(1-\frac{2k}{n}\right) \right)^{\ell}= \sqrt{n} \left( 1- \frac{k}{n} \right)^{\ell},$$
because $f$ is an eigenfunction of the Markov chain corresponding to the eigenvalue $1- \frac{k}{n} .$
Again under the convolution measure, we have that
 $$\operatorname{Var} \{ f(Z_{\ell}) \} =$$
 $$\frac{n}{n} + \frac{n(n-1)}{n} \left( 1+ \frac{2k^2-2kn}{n^2-n}\right)^{\ell} - n \left( 1- \frac{k}{n}  \right)^{2 \ell}=$$
 $$1 +(n-1) \left( 1- \frac{2kn-2k^2}{n^2-n}\right)^{\ell} - n \left( 1- \frac{k}{n}  \right)^{2 \ell}.$$
Recall that the first three (non-normalized)  eigenfunctions of this Markov chain are
$$f_0(x)=1, f_1(x)=1-\frac{2x}{n} \mbox{ and }f_2(x)= 1 -\frac{4x}{n-1} + \frac{4x^2}{n^2-n} .$$

By direct computation, we have that $f^2_1(x)= \frac{1}{n} f_0(x) + \frac{n-1}{n} f_2(x)$. Combining this and the fact that $f_2$ corresponds to the eigenvalue $1+ \frac{2k^2-2kn}{n^2-n}$ gives the claimed variance. Now, take $\ell$ of the form $\frac{n}{2k}\log n -c\frac{n}{k}$ (where $c>0$).
\begin{enumerate}
\item  First case to be considered is $k =  \frac{n}{d}  $ where $d$ is a constant. 
$$ \mathbb{E}( f(Z_{\ell})) =  \sqrt{n} \left( 1- \frac{k}{n} \right)^{\ell} =\sqrt{n} \left( 1- \frac{1}{d} \right)^{\ell} $$

and 

$$\operatorname{Var} \{ f(Z_{\ell})\} = 1 +(n-1) \left( 1- \frac{2kn-2k^2}{n^2-n}\right)^{\ell} - n \left( 1- \frac{k}{n}  \right)^{2 \ell } $$
$$ =   1 +(n-1) \left( 1-  \frac{2n(1- \frac{1}{d})}{d(n-1)}   \right)^{\ell} - n \left( 1 - \frac{1}{d}  \right)^{2 \ell} $$ $$\leq 1 +(n-1) \left( 1-  \frac{2(1- \frac{1}{d})}{d}   \right)^{\ell} - n \left( 1 - \frac{1}{d}  \right)^{2 \ell}  $$ $$=1 +(n-1) \left( \frac{d-1}{d}  \right)^{2 \ell} - n \left(  \frac{d-1}{d}  \right)^{2 \ell} .$$

In this case if $\ell =1/2 \log_{d/(d-1)} (n) - c$, then 
$$ \mathbb{E}\{f(Z_{\ell})\}\geq  \left( \frac{d}{d-1} \right)^c$$
so if $0<c< \frac{1}{4} \log_{d/(d-1)} (n)$ the expectation $ E \{ f(Z_{\ell})\}$ can get big while
$$\operatorname{Var} \{ f(Z_{\ell})\} \leq  2 .$$

\item   If there is $0<\varepsilon<1$, so that $k=O(n^{\varepsilon})$, then the mean  becomes

$$\mathbb{E}\{ f(Z_{\ell})\}= \exp \left(c +  O\left(\frac{k\log n}{n} \right)+O\left(\frac{ ck}{n}\right)\right), $$
which means that for $0<c <  \frac{1}{4} \log (n/k)$ this expectation is big.
Similarly for the variance one gets
\begin{equation*}
\begin{split}
& \operatorname{Var} \{ f(Z_{\ell}) \}   \\ 
& =1 +(n-1) \exp \left( c \frac{2n-2k}{n-1} - \frac{n-k}{n-1} \log n + O\left(\frac{k}
{n}  \log n \right) + O\left( \frac{ck}{n} \right)  \right) \\
& - \exp \left( 2c + O\left( \frac{k\log n}{n} 
\right) +O\left( \frac{ck}{n} \right)  \right)  \\
& \leq  1+O\left( \frac{1}{n}\right) + e^{2c} \left( O\left( \frac{k\log n} {n} \right) +O\left(\frac{ck}{n} \right) \right).
\end{split}
\end{equation*}
Therefore the variance is uniformily bounded for $0 \leq c < \frac{1}{4} \log (n/k$).

 \end{enumerate}
In both cases, Chebyshev's inequality gives that for the set $A_{\alpha} = \{ x: |f(x)| \leq \alpha \}$, we have that
$$U(A_{\alpha} ) \geq 1- \frac{1}{\alpha^2} \mbox{ and } P^{*\ell}(A_{\alpha}) < \frac{B}{(e^{2c} - \alpha)^2}, $$
where $B$ is uniformly bounded when $0 \leq c \leq \frac{1}{4} \log n$.
Therefore, we have
$$||P^{*\ell}_{id} - U||_{T.V.} \geq 1- \frac{1}{\alpha^2} - \frac{B}{(e^{2c}- \alpha)^2}.$$
Now, take $\alpha= \frac{e^{2c}}{2}$, which finishes the proof.

\end{proof}

\section{Fourier Transform Arguments}\label{eig}

In this section, a different approach is introduced. It combines the representation theory of the hypercube and the Fourier transform to provide a bound for the mixing time. All of the irreducible representations of the hypercube are one dimensional and they are indexed by
$\textbf{z} \in (\mathbb{Z}/ 2 \mathbb{Z})^n$ in the following way:
$$\rho_{\textbf{z}}(\textbf{w})= (-1)^{\textbf{z} \cdot  \textbf{w}},$$
 where $\textbf{z} \cdot  \textbf{w}$ is the inner product of $\textbf{z}, \textbf{w} \in (\mathbb{Z}/2\mathbb{Z})^n$.
The Fourier transform of a probability $P$ at a representation $\rho$ is defined as
$$\widehat{P}(\rho)= \sum_{g \in (\mathbb{Z}/2 \mathbb{Z})^n} P(g) \rho(g),$$
which in our case means
\begin{equation}\label{eigen}
\widehat{P}(\rho_{\textbf {z}})= p +(1-p) \sum^j_{a=0} (-1)^a \frac{{j \choose a} {n-j \choose k-a}}{{n \choose k}}= p +(1-p) K^n_j(k),
\end{equation}
 where $j$ is the number of ones that $\textbf{z}$ has and $K^n_j(k)$ is the $j$th Krawtchuck polynomial evaluated at $k$.

In Chapter 3 of \cite{Pbook}, one can find the Upper Bound Lemma (Lemma 1 in the book) which shows how using the Fourier transform of the representations of a group to find an upper bound for the mixing time of a walk on the group. More precisely, the upper bound lemma in the case of the hypercube (or in general for $(\mathbb{Z}/p \mathbb{Z})^n$ says:
 
 \begin{lemma}(Upper Bound Lemma)\label{ubl}
For a random walk on the hypercube, after $\ell$ steps, we have that
 $$4\lVert P_x^{*\ell} -U\rVert ^2_{T.V.} \leq \sum_{ \textbf{z} \neq 0} (\widehat{P}(\rho_{\textbf{z}}) )^{2\ell},$$
 for all $x \in (\mathbb{Z}/ 2 \mathbb{Z})^n$.
 \end{lemma}
Lemma \ref{ubl} will be used for the case $k=\frac{n}{2}$ and for a the walk generated by $Q$ as defined in \eqref{q}.

\section{The case k=n/2}

In the case where $n$ is even with $n=2k$  where $k$ is a positive, odd integer, the following facts hold:

\begin{lemma}
For $k= \frac{n}{2}$ the Fourier Transform of representation $\rho_{\textbf{z}}$ is given by
$$\widehat{P}(\rho_{\textbf{z}})= \begin{cases} \frac{1}{2}, & \mbox{if j is odd;} \\ \frac{1}{2} +\frac{(-1)^i {\frac{n}{2} \choose i}}{2{n \choose 2i}}, & \mbox{if j=2i,}  \end{cases}$$
where $j$ is the number of ones that $\textbf{z}$ has.  
\end{lemma}

\begin{proof}

Acccording to Koekoek and Swarttouw  in \cite{Koekoek} the $j$th Krawtchuck polynomial satisfies the following recurrence relation:
\begin{equation}\label{koSwa} 
-k K^n_j(k)= \frac{1}{2} (n-i)K^n_{j+1}(k) -\frac{n}{2} K^n_j(k) +\frac{n}{2}K^n_{j-1}(k) .
\end{equation}
If $k=\frac{n}{2}$, then equation \eqref{koSwa} gives that
$$K^n_j \left(\frac{n}{2} \right)= \begin{cases} 0, & \mbox{if } j\mbox{ is odd;} \\ \frac{(-1)^i {\frac{n}{2} \choose i}}{{n \choose 2i}}, & \mbox{if }j=2i  , \end{cases}$$
since that $K^n_0(k)=1$ and $K^n_1(k)=1$.
\end{proof}

The next step is to bound the eigenvalues and  use the Upper Bound Lemma  to actually get an upper bound for the $L^2$ norm.
\begin{lemma}\label{bound_eig}
For every representation $\rho_{\textbf{z}}$ where $\textbf{z} \neq \textbf{0}$,
\begin{equation}
|\widehat{P}(\rho_{\textbf{z}})| \leq \frac{3}{4}.
\end{equation}
\end{lemma}

\begin{proof}
Let $j$ be the number of ones $\textbf{z}$ has. If $j$ odd the theorem holds because $\widehat{P}(\rho_{\textbf{z}})= \frac{1}{2}$. If $j=2i$ the quantity $\frac{1}{2} +\frac{(-1)^i {\frac{n}{2} \choose i}}{2{n \choose 2i}} $ is the main concern.
For $i$ odd, the second term is negative but bigger than $- 1/2$ therefore the quantity is positive  less than $\frac{1}{2}$. For $i$ even, it turns out that $c_i= \frac{{\frac{n}{2} \choose i}}{{n \choose i}}$ is maximized for $i= \frac{n}{2}-1$ and at most $ \frac{1}{n-1}$ by a simple argument. Thus, for $i$ even, $\frac{1}{2} +\frac{(-1)^i {\frac{n}{2} \choose i}}{2{n \choose 2i}} \leq \frac{3}{4}$.
\end{proof}

\begin{theorem}
For $k=\frac{n}{2}$ and for $\ell= \frac{n \log2 - \log \varepsilon}{\log \frac{4}{3}} $ with $0 < \varepsilon <1$,
$$4\lVert  P_x^{* \ell} -U\rVert ^2_{T.V.} < \varepsilon. $$
for all $x \in (\mathbb{Z}/ 2 \mathbb{Z})^n $.
\end{theorem}

\begin{proof}
After having computed the Fourier transform of each representation and bounded them in Lemma \ref{bound_eig},  the upper bound lemma gives that
$$4\lVert  P_x^{* \ell} -U\rVert ^2_{T.V.} \leq \sum^n_{j=1} {n \choose j} b_j^{2\ell} \leq 2^n \left(\frac{3}{4}\right)^{2\ell} \leq \varepsilon.$$
\end{proof}

\begin{remark}
Notice that for the $\ell^2$ norm, we have that
$$|G| \lVert  P_x^{* \ell} -U\rVert ^2_2 \geq \left( \frac{1}{2} \right)^{2\ell} \sum_{\mbox{i is odd}} {n \choose i}= 2^{\frac{n-1}{2}-2\ell} ,$$
so indeed the $L^2$ norm cannot give a better upper bound for the mixing time. However, the $L^1$ norm may still be small for smaller $\ell$.
\end{remark}

\section{Similar Random Walks}\label{more}
This section focuses on similar random walks on $(\mathbb{Z}/ m \mathbb{Z})^n$. It uses Fourier Transforms and comparison theory methods to bound their  mixing times. 

\subsection{A random walk on $(\mathbb{Z}/ m \mathbb{Z})^n$ }
Consider the walk generated by the measure

\begin{equation*}
Q(a_{i_1} e_{i_1} +a_{i_2} e_{i_2}+ \ldots + a_{i_k} e_{i_k}  )= \frac{1}{{n \choose k}m^k},
\end{equation*}
where $a_{i_j} \in \mathbb{Z}/ m \mathbb{Z}$ and $\{ i_1, i_2, \ldots i_k\} \subset \{1,2,\ldots, n\}$.
\begin{proof}[Proof of Theorem \ref{gen}]
Let $\rho_{a}$ denote a representation of $(\mathbb{Z}/ m \mathbb{Z})^n$, where $a\in (\mathbb{Z}/ m \mathbb{Z})^n$. If $g = (g_1,\ldots, g_n) \in (\mathbb{Z}/ m \mathbb{Z})^n$ then
$$\rho_{a}(g) =e^{\frac{2 \pi i \sum^n_{j=1} a_jg_j }{m}} . $$
Then the Fourier transform of $Q$ at this representation  with respect to $Q$ is
$$\widehat{Q}(\rho_{a})= \frac{1}{{n \choose k}m^k} \sum_{|g| \leq k} e^{\frac{2 \pi i \sum^n_{j=1} a_jg_j }{m}} , $$ 
where $|g|$ denotes the number of positions that $g_i$ is not equal to zero.
Now $$\sum_{b \in \mathbb{Z}/ m \mathbb{Z}}  e^{\frac{2 \pi i a_j b}{m}}= m \delta_{0, a_j},$$ 
with $\delta_{0, a_j}= 1$, if $a_j=0$ and $\delta_{0, a_j}= 0$ otherwise. If  $n- |a| \geq k $, then
$$\widehat{Q}(\rho_{a})= \frac{{n- |a| \choose k}}{{n \choose k}},$$
otherwise $\widehat{Q}(\rho_{a})=0$. These are the eigenvalues of the random walk that are not equal to $1$.

Now notice that all of the eigenvalues are non-negative and in particular
\begin{align*}
 \frac{{n- |a| \choose k}}{{n \choose k}}&= \left(1-\frac{k}{n} \right)\left(1-\frac{k}{n-1} \right)\ldots \left(1-\frac{k}{n-|a|+1}\right) \\
 & \leq e^{-k \sum^n_{j=n-|a|+1} \frac{1}{j}} \\
 &   \leq e^{- k \log \frac{n+1}{n -|a|+1}} = \left( 1- \frac{|a|}{n+1}\right)^k.
 \end{align*}

Then, the Upper Bound Lemma (Lemma \ref{ubl}) gives that 
$$4||Q_x^{*\ell}-U||^2_{T.V.} \leq \sum^{n-k}_{j=1} {n \choose j } (m-1)^j \left( 1- \frac{j}{n+1}\right)^{2kl}$$ $$  \leq \sum^{n-k}_{j=1} \frac{n^j}{j!} (m-1)^j \left( 1- \frac{j}{n+1}\right)^{2kl} \leq e^{-c},$$
where $\ell=\frac{n+1}{2k} \log (mn) + \frac{c(n+1)}{2k}$.
\end{proof}

\begin{remark}
Notice that the first time that all coordinates have been touched, is a strong stationary time, which implies that the total variation distance needs order $\frac{n}{k} \log n $ steps to get small. To see this one can imitate the calculation for the coupon collector problem as presented in Lemma $2$ of \cite{A-D}. Further, notice that
$$|G|\lVert Q_x^{*\ell}-U\rVert ^2_2 \geq n (m-1) \left( 1- \frac{1}{n}\right)^{2kl},$$
which means that the $L^2$ norm needs at least $\frac{n}{2k} \log (mn) + \frac{c(n+1)}{2k}$ steps to get small.
Therefore, there is a gap between the separation distance and the $L^2$ norm mixing times.
\end{remark}

\subsection{Comparison Theory Application}\label{comp}
Comparison theory can help provide an upper bound for the following example:
\begin{example}
With notation as in Theorem \ref{gen}, consider the case $k=1$. Then, 
\begin{equation*}
Q(be_{i}) = \frac{1}{mn},
\end{equation*}
for $b\in \mathbb{Z}/ m \mathbb{Z} $ and $1 \leq i \leq n$. The walk suggests to pick a coordinate at random and randomize it.
Theorem \ref{gen} states that if $\ell= \frac{1}{2}((n+1) \log (m n) +c(n+1))$, then for every $x \in\mathbb{Z}/ m \mathbb{Z}$, we have that
$$4||Q_x^{*\ell}-U||^2_{T.V.} \leq e^{-c} .$$
Comparison theory gives the following theorem for the mixing time of the random walk generated by
\begin{equation*}
\text{\~{P}}(\pm e_1)= \frac{1}{4n}, \text{\~{P}}(id)= \frac{1}{2}.
\end{equation*}
and let $\text{\~{P}}(x,gx)= \text{\~{P}}(g) $, for every 
$x,g \in (\mathbb{Z}/ m \mathbb{Z})^n$.
\begin{theorem}
Let $\ell= \frac{1}{2}m^2 ((n+1) \log (mn) +c(n+1)$, then  for every $x \in (\mathbb{Z}/ m \mathbb{Z})^n$, we have that
$$4||\text{\~{P}}_x^{*\ell}-U||^2_{T.V.} \leq \left(1+\frac{1}{n^n} \right)e^{-c}.$$
\end{theorem}

\begin{proof}
Let $S= \{ \pm e_j, id\}$ and $S'= \{ be_j, b \in \mathbb{Z}/ m \mathbb{Z}\}$. According to P.\ Diaconis and L.\ Saloff-Coste \cite{Compare} if we represent each $\textbf{z} \in S'$ as a product of elements of $S$ that has odd length and
$$A= \max_{s \in S} \frac{1}{\text{\~{P}}(s)} \sum_{\textbf{z} \in S'} \lVert \textbf{z} \rVert  N(\textbf{z},s) Q(\textbf{z}),$$
where $\lVert \textbf{z}\rVert $ is the length of this representation and $N(\textbf{z},s)$ is the number of times $s$ is inside the representation of $\textbf{z}$, then
$$4\lVert \text{\~{P}}_x^{*\ell}-U \rVert^2_{T.V.} \leq m^n e^{-\ell/A} + m^n \lVert Q^{*\ell/2A}-U \rVert_2^2 . $$
An easy argument shows that $A \leq \max \{m^2 , \frac{2}{m} +2m \}= m^2 $. Therefore, if $\ell=\frac{m^2}{2} ((n+1) \log (m n) +c(n+1)) $,
$$4 \lVert \text{\~{P}}^{*\ell}_x-U \rVert^2_{T.V.} \leq \frac{e^{-c}}{n^n}+ e^{-c} . $$
\end{proof}
\end{example}

\acks
The author would like to thank Persi Diaconis and Graham White for the helpful comments they provided.

\bibliographystyle{apt}
\bibliography{final}
\end{document}